\documentclass[pdflatex,sn-mathphys-num]{sn-jnl}

\usepackage{graphicx}%
\usepackage{multirow}%
\usepackage{amsmath,amssymb,amsfonts}%
\usepackage{amsthm}%
\usepackage{mathrsfs}%
\usepackage[title]{appendix}%
\usepackage{xcolor}%
\usepackage{textcomp}%
\usepackage{manyfoot}%
\usepackage{booktabs}%
\usepackage{algorithm}%
\usepackage{algorithmicx}%
\usepackage{algpseudocode}%
\usepackage{listings}%

\theoremstyle{thmstyleone}%
\theoremstyle{thmstyletwo}%

\theoremstyle{thmstylethree}%

\newtheorem{defi}{Definition}[section]
\newtheorem{theor}[defi]{Theorem}

\newtheorem{lem}[defi]{Lemma}

\makeatletter
\newcommand{\Jac}{\mathop{\operator@font Jac}}
\newcommand{\rang}{\mathop{\operator@font rang}}
\newcommand{\sol}{\mathop{\operator@font sol}}
\newcommand{\argmin}{\mathop{\operator@font argmin}}
\makeatother

\begin{document}
	
\title[Dimensional reduction associated with unchanged-direction trajectories in flat fiber Robertson-Walker spacetimes]{Dimensional reduction associated with unchanged-direction trajectories in flat fiber Robertson-Walker spacetimes}


\author[1]{\fnm{Daniel} \sur{de la Fuente}}\email{fuentedaniel@uniovi.es}

\author[2]{\fnm{Rafael M.} \sur{Rubio}}\email{rmrubio@uco.es}
\equalcont{These authors contributed equally to this work.}

\author*[2]{\fnm{Jose} \sur{Torrente-Teruel}}\email{jtorrente@uco.es}
\equalcont{These authors contributed equally to this work.}

\affil[1]{\orgdiv{Departament of Mathematics}, \orgname{University of Oviedo}, \orgaddress{\street{Escuela Polit\'ecnica de Ingenier\'ia}, \city{Gij\'on}, \postcode{33003}, \country{Spain}}}

\affil*[2]{\orgdiv{Departamento de Matem\'aticas}, \orgname{University of C\'ordoba}, \orgaddress{\street{Campus de Rabanales}, \city{C\'ordoba}, \postcode{14071}, \country{Spain}}}


\abstract{In this article, we establish a new dimensional reduction result in (Galilean or Lorentzian) Robertson--Walker spacetimes with Euclidean fiber. In this context, we prove that every unchanged-direction observer can be embedded into a three-dimensional totally geodesic timelike submanifold.}




\maketitle

\section{Introduction}
It is well known that General Relativity provides the most appropriate theoretical framework to describe the geo\-metry of spacetime and the geometric nature of gravitation, in which Lorentzian geometry constitutes the mathema\-tical framework. Since its inception at the beginning of the twentieth century, its predictions have been continuously confirmed and its foundations progressively consolidated. Newton-Cartan theory was developed a few years later \cite{Cartan1,Cartan2}, providing a geometrical formulation of spacetime and Newtonian gravity. Despite being incompatible with physical experiments involving velocities close to the speed of light, it remains useful and practical for describing most physical phenomena, largely due to its intuitive nature and operational simplicity. Robertson--Walker cosmological models are considered as a benchmark in both settings \cite{ONeill1983,Muler1983} (see also \cite{AliasRomeroSanchez1995,GGRW}), describing the large-scale evolution of the Universe under the general assumptions of homogeneity and isotropy, according to the cosmological principle. Although both theories admit an expanding (or contracting) behaviour of spacetime, they are not equivalent; for instance, flat restspaces appear naturally and intrinsically in the Newtonian formulation, while this is not the general case in their relativistic counterpart.

In both Galilean and Lorentzian models, the trajectories of galaxies (or, more precisely, of the constituent particles) are encoded by curves in spacetime. Among these, trajectories with vanishing proper acceleration (geodesics) are crucial, as they model free falling particles. However, accelerated trajectories are relevant, given that forced systems emerge in a variety of contexts, including the measurement of specific gravitational effects and the analysis of rocket travel. In this regard, one of the most surprising theoretical discoveries is the Unruh effect \cite{Unruh1976}. In our case, we focus on a generalization of uniformly accelerated trajectories \cite{UAM, GalileanUA}, known as unchanged-direction (UD) motion, which correspond to observers measuring a proper acceleration with constant (parallel) direction \cite{UDM, GalileanUD}. 

In this article, we aim to establish an Erbacher-type dimensional reduction of UD observers immersed in a Robertson--Walker spacetime with flat fiber (i.e., isometric to an open subset of Euclidean space). We continue the work presented in \cite{DaniRafaJose_Erbacher}, where such a reduction was proved in product spacetimes, so that all spatial slices had the \emph{same} constant sectional curvature. This assumption provided sufficient symmetry to obtain a dimensional reduction: any UD observer can be embedded into a totally geodesic timelike $3$-dimensional submanifold. In the present work, we prove an analogous statement for Robertson--Walker spacetimes with fiber $\mathbb{R}^n$. In this case, they also decompose as a product manifold, but the constant sectional curvature may differ between restspaces, thereby breaking the symmetry that was crucial in the proofs of both~\cite{DaniRafaJose_Erbacher} and Erbacher's original article~\cite{Erbacher1971} for semi-Riemannian space forms. Moreover, whereas in~\cite{DaniRafaJose_Erbacher} the Lorentzian and Galilean cases were treated in a parallel manner, here they diverge and require separate treatments (see Section \ref{section3}). Remarkably, our contribution relate again to the dual statement of the main result in both Galilean and Lorentzian frameworks.

In the following, the main result of this letter is stated. From a physical point of view, this implies that the study of the kinematics of a particle following an unchanged-direction trajectory can be reduced to a spacetime of only three dimensions:

\begin{theor}\label{maintheorem}
	In any (Galilean or Lorentzian) Robertson--Walker spacetime with flat fiber, every unchanged-direction observer can be embedded into a totally geodesic timelike submanifold of at most three dimensions.
	
\end{theor}

In the next section, we briefly review the necessary elements of Galilean and Lorentzian geometries, as well as the definitions of Robertson--Walker spacetimes in both settings. We also recall the unchanged direction notion together with its characterization in terms of the initial velocity and acceleration. In Section \ref{section3}, we present the proof of Theorem~\ref{maintheorem}, which differs depending on whether the framework is relativistic or non-relativistic.

\section{Setup}
\label{section:preliminaries}
A connected smooth manifold $M^{n+1}$ is called a \emph{spacetime} if it is endowed with one of the following two geometric structures:
\begin{itemize}

	\item A \emph{Galilean} structure, $(M,\Omega, h, \nabla)$, where $\Omega\in \Lambda^1(M)$ is a non-vanishing one-form, $h$ is a positive definite metric defined on the vector bundle given by the distribution $\text{An}\Omega = \{v\in TM\, : \, \Omega(v)=0\}$, and $\nabla$ is a Galilean connection, i.e., an affine connection on $M$ parallelizing $(\Omega,h)$ \cite{Bernal}. 
	Given $p\in M$, $(\text{An}\Omega_p,h_p)$ is the \emph{absolute space} at $p$ and $\Omega_p$ is the \emph{absolute clock} at $p$. 
	In line with standard approaches in physics, we shall assume that $\nabla$ is torsion-free, which implies that $d\Omega = 0$ \cite[Lemma 13]{Bernal}. Consequently, the distribution $\mathrm{Ann}\,\Omega$ is integrable.
	The associated spacelike leaves $\Sigma$ are called \emph{(absolute) restspaces} and the spacetime locally admits absolute simultaneity (see \cite{GFPR} and references therein). 
	
	\item A \emph{Lorentzian} structure $(M,g)$, where $g$ is a Lorent\-zian metric with signature $(-,+,\ldots,+)$ and a fixed time orientation. In this case, we consider the Levi--Civita connection $\nabla^g$, which is torsionless. This is the geometric framework in which relativistic physics is developed, where time or simultaneity has no absolute character but is relative to each field of
	observers. In particular, a field of observers $Z$ is locally synchroni\-zable with respect to its proper time if its orthogonal distribution $Z^\perp$ is integrable,
	and so, each of the spacelike leaves is a restspace associated with $Z$ \cite[p. 358]{ONeill1983}.
	
\end{itemize}

Each point of $M$ corresponds to a spacetime \emph{event}. A vector $v \in T_pM$
is said to be \emph{timelike} if $\Omega(v)\neq 0$ or $g(v,v)<0$  and \emph{spacelike} if $\Omega(v) = 0$ or $g(v,v)>0$, respectively. In the Lorentzian setting, $v$ is \emph{lightlike} if $g(v,v) = 0$, but this notion does not make sense in the Galilean context (infinite speed of light assumption). A timelike vector is \emph{future-pointing} if $\Omega(v)>0$ or $v$ is in the future cone of $(T_pM, g_p)$ (\emph{past-pointing} otherwise) and it is \emph{unitary} if $\Omega(v) = 1$ or $g(v,v) = -1$, respectively. Note that we assume time-oriented Lorentzian structure. More details may be found in \cite{Bernal, SachsWu1977}.

An \emph{observer} is a smooth curve $\gamma\colon J\subset \mathbb{R} \to M$ such that its tangent vector field (\emph{proper velocity}) is unitary and future-pointing timelike. Its parameter $s\in J$ corresponds to the \emph{proper time}, i.e., the time measured by its clock. The physical space observed by $\gamma$ at a proper time $t$ is the associated absolute restspace at $\gamma(t)$ (Galilean case) or the associated local restspace in $\gamma(t)$ with respect to $\gamma'(t)$ (Lorentzian case). Its covariant derivative $\frac{D\gamma\,'}{dt}$ is the (\emph{proper}) \emph{acceleration}, which can be empirically measured by an accelerometer \cite[Sect. 1]{UAM}.
A vector field $Z\in \mathfrak{X}(M)$ is called \emph{field of observers} if its integral curves are observers. Then, for each $X\in\mathfrak{X}(M)$, the spacelike projection along $Z$, $P^Z$, is defined as: 
\begin{equation*}
	\begin{aligned}
		P^Z X ={} X - \Omega(X) Z  \quad \text{(Galilean case)},\qquad
		P^Z X ={} X - g(X,Z) Z \quad \text{(Lorentzian case)}.
	\end{aligned}
\end{equation*}

A submanifold $S \subset M$ is said to be totally geodesic if every geodesic
of $M$ whose initial point lies in $S$ and whose initial velocity is tangent
to $S$ remains entirely contained in $S$.

\subsection{ Robertson--Walker spacetimes with flat fibers}
In this section, we review the definitions and fundamental equations for the connection in (Galilean and Lorentzian)
Robertson?Walker spacetimes, which constitute the ambient spaces in this work.
These structures appear as the ambient framework in Theorem \ref{maintheorem}. 

Consider $I\subseteq \mathbb{R}$ an open real interval, $U\subseteq\mathbb{R}^n$ be an open subset of $\mathbb{R}^n$, $\langle\cdot,\cdot\rangle$ the Euclidean metric of $\mathbb{R}^n$ and a positive function $f\in C^{\infty}(I)$.
\subsubsection{Galilean Robertson-Walker spacetimes}
A Galilean spacetime $(M,\Omega, h, \nabla)$ is called \emph{Galilean Robertson-Walker} (RW) spacetime with flat restspace if $M=I\times U$, $\Omega = d \pi_I$, $g$ is the restriction to the distribution $\text{An}\Omega$ of the following (degenerate) metric on $M$:
\begin{equation}
	\overline{h} = (f\circ\pi_I)^2 \pi_U^* \langle\cdot,\cdot\rangle\,,
\end{equation}
where $\pi_I, \pi_U$ are the canonical projections onto the open interval $I$ and $U$, respectively, and $\nabla$ is the unique symmetric Galilean connection on $M$ such that
\begin{equation}\label{eq:GR conditions}
	\nabla_{\partial_t} \partial_t=0 \quad \text { and } \quad \operatorname{Rot} \partial_t=0\,,
\end{equation}
where $\partial_t=\partial / \partial t$ is the global coordinate vector field associated with $t:=\pi_I$. Abusing the notation, we identify $\overline{h}$ and $h$, and we write $h=f^2 \pi_U^* \langle\cdot,\cdot\rangle$.

Explicitly, we can give the following expression for $\nabla$.
If $V,W\in \mathfrak{X}(U)$, then
\begin{equation}\label{eq:GRW connection}
	\begin{split}
		\nabla_{\partial_t} \overline{V} ={} \nabla_{\overline{V}} \partial_t = u'\, \overline{V}\,,\qquad
		\nabla_{\overline{V}} \overline{W} ={} \overline{\nabla^{\mathbb{R}^{n}}_{V} W} ,
	\end{split} 
\end{equation}
where $u:=\log f$ and where the bar over a vector  in $T_xU$ denotes its lift to a point
$p=(t,x)\in M$, namely, the unique vector in $T_pM$, whose projection via $\pi_U$ is tangent to 
$\mathbb{R}^n$ and $\nabla^{\mathbb{R}^n}$ is the standard Levi-Civita connection in $\mathbb{R}^n$.

\subsubsection{Relativistic Robertson-Walker spacetimes}

A \emph{Lorentzian Robertson-Walker} (RW) spacetime with flat fiber is a product manifold $I\times U$ endowed with the following Lorentzian metric
\begin{equation}\label{metrica}
	g=\pi_{I}^{*}(-dt^{2})+f^{2}(\pi_{I})\pi_{U}^{*}(\langle\,,\rangle)\equiv -dt^{2}+f^{2}\langle\;,\;\rangle\,.
\end{equation}
Thus, $(M, g)$ is a Lorentzian warped product with base $(I,-dt^2)$, fiber $(U,\langle\cdot,\cdot\rangle)$ and warping function $f$, and it is usually denoted by $I\times_{f}U$. 

\noindent If we also denote by $u=\log f$, then the Levi--Civita connection of $M$, $\nabla$, satisfies 
\begin{equation}\label{connection}
	\begin{aligned}
		\nabla_{\partial_t}\partial_t = 0,\qquad
		\nabla_{\partial_t}\overline{V}
		= \nabla_{\overline{V}}\partial_t = u'\,\overline{V}, \qquad
		\nabla_{\overline{V}}\overline{W}
		= \overline{\nabla^{\mathbb{R}^n}_V W}
		- u'\,g( \overline{V},\overline{W})\,\partial_t ,
	\end{aligned}
\end{equation}
for all $V,W \in \mathfrak{X}(U)$.


\subsection{Unchanged-direction trajectories}

Finally, we recall the notion of UD motion and show a global integration formula for UD observers, following the results in \cite{UDM} and \cite{GalileanUD}. 

Let $M$ be a spacetime and let $\gamma\colon J\to M$ be an observer. It obeys a \emph{proper Unchanged Direction} (proper UD) motion if either $\gamma$ is free falling or the proper acceleration does not vanish at any instant and its direction, $U_\gamma=\Big|\frac{D\gamma'}{dt}\Big|^{-1}\frac{D\gamma'}{dt}$, remains constant along its trajectory, i.e.,

\begin{equation}
	\label{eq:UD}      
	\frac{DU_\gamma}{dt} = 0\,\quad \text{(Galilean case)},\qquad  \frac{\widehat{D}U_\gamma}{dt} = 0\,\quad \text{(Lorentzian case),}
\end{equation}
where $\frac{D}{dt}$ and $\frac{\widehat{D}}{dt}$ are the Galilean and Fermi-Walker
covariant derivatives, respectively (for details, see \cite[Prop. 2.2.2]{SachsWu1977}).  

A generalization of UD motion allowing the proper acceleration to vanish was developed in \cite[Section II]{UDM} and \cite[Section 3]{GalileanUD}). Here we recall a useful characterization of general UD trajectories that will be used throughout this paper. Given a prescribed acceleration function $a\colon J \to \mathbb{R}$, with $a(t_0) \ne0$ for some $t_0\in J$, the following initial values can be posed in $t_0$,

\begin{equation}   \label{eq:initial values}     
	\begin{aligned}
		\gamma(t_0) = p \in M,\qquad  
		\gamma'(t_0) = v \in T_pM,  \qquad
		\frac{D\gamma'}{dt}(t_0) = a(t_0) w,
	\end{aligned}
\end{equation}
\noindent where $\Omega(v) = 1$, $\Omega(w)=0$ and $\langle w,w\rangle=1$ (Galilean case) or $-g(v,v)=1=g(w,w)$ (Lorentzian case). An observer $\gamma\colon J\to M$ obeys an UD motion with acceleration $a(t)$ and initial values (\ref{eq:initial values}) if and only if it satisfies  
\begin{subequations}
	\label{eq:global expression}
	\begin{align}
		\label{eq:global expression G}
		\gamma'(t) ={}& L(t) + \alpha(t) N(t)\,  \qquad  \text{(Galilean case),} \\
		\label{eq:global expression L}
		\gamma'(t) ={}& \cosh{(\alpha(t))} L(t) + \sinh{(\alpha(t))} N(t) \,\quad \text{(Lorentzian case)},
	\end{align}
\end{subequations}
where $\alpha(t):=\int_{t_0}^t a(s) ds$ and $L,N$ are parallel vector fields along $\gamma$ such that $L(t_0)=v$ and $N(t_0)=w$.

\section{Proof of Theorem \ref{maintheorem} }\label{section3}

\subsection{Galilean case}

Let $(M=I\times U,dt,h=f^2\pi_{U}^{*}\langle\cdot,\cdot\rangle,\nabla)$ be a Galilean RW spacetime, $U\subset\mathbb{R}^n$, and let $S\subset M$ be a smooth connected submanifold. For each $p\in S$, define the spacelike vector subspace in $T_pS$, 
\begin{equation*}
	\mathcal{R}_p := T_pS\cap \text{An}_p(dt)\,,
\end{equation*}
and denote by $\mathcal{R}_p^\perp$ the spacelike orthogonal subspace with respect to the absolute space at $p$, $(\text{An}_p(dt), h_p)$.
Assume that $S$ is a timelike submanifold, i.e., $T_pS \ne \mathcal{R}_p$, for all $p\in S$.

Hence, given $X,Y\in \mathfrak{X}(S)$, the \emph{induced Galilean connection} can be defined on $S$, $\nabla^S_X Y$, satisfying
$$\nabla_X Y = \nabla^S_X Y + \mathrm{II}(X,Y),$$
where $\nabla^S_X Y \in \mathfrak{X}(S)$ and $\mathrm{II}(X,Y) \in \Gamma(\mathcal{R}^\perp)$. Therefore, we obtain an induced connection on $S$, $\nabla^S$, and a \emph{second fundamental form}, $\mathrm{II}$, analogous to the corresponding notion in semi-Riemannian geometry. Making use of $\mathrm{II}$, we can characterize the totally geodesic submanifolds.
\begin{lem}
	Given a Galilean Robertson-Walker spacetime $M=I\times U$, $U\subset\mathbb{R}^n$, a $k$-dimensional submanifold $S$ of $M$ is totally geodesic in $M$ if and only if either $S$ is a totally geodesic submanifold of a leaf of $M$, or $S$ is a subset of a k-plane tangent to $\partial_t$.
\end{lem}
\begin{proof}
	Since the connection $\nabla$ of each restspace agrees with the connection of the fiber $U$, a submanifold $S \subset \{t\}\times U$ is totally geodesic if and only if it is totally geodesic in that slice. Henceforth, we assume that $S$ is timelike. Then, given $X,Y\in \mathfrak{X}(S)$, there exist $a_0,a_i,b,b_j\in\mathcal{C}^\infty(S)$, $i,j=1,\cdots,n$, such that
	\begin{subequations}
		\begin{align*}
			&X=a_0\partial_t+\widehat{V} \quad\text{with}\quad \widehat{V}:=P^{\partial_t}X=\sum_{i}a_i\overline{\partial_{x_{i}}},\\\nonumber\qquad
			&Y=b\big(\partial_t+\widehat{W}\big) \quad\text{with}\quad \widehat{W}:=\sum_{j}b_j\overline{\partial_{x_j}},\nonumber
		\end{align*}
	\end{subequations}
	where $\lbrace{\partial_{x_i}\rbrace}_i$ are the standard local coordinate vector fields of $U\subset\mathbb{R}^n$. We can derive
	\begin{equation}\label{eq10}
		\nabla_X Y=X(b)\big(\partial_t+W\big)+b\nabla_X \big(\partial_t+W\big).
	\end{equation}
	Denoting $V=\sum_{i}a_i\partial_{x_{i}}$ and $W=\sum_{j}b_j\partial_{x_j}$, the second term can be computed through (\ref{eq:GRW connection}) as follows
	\begin{eqnarray}
		\nabla_X(\partial_t+\sum_j b_j\overline{\partial_{x_j}})&=&\nabla_V\partial_t+a_0\sum_j (\partial_tb_j)\overline{\partial_{x_j}}+a_0\sum_j u'b_j\overline{\partial_{x_j}}\nonumber\\
		&=&u'\widehat{V}+\overline{\nabla^{\mathbb{R}^n}_V W}+a_0\big(\sum_j(\partial_tb_j+u'b_j)\overline{\partial_{x_j}}\nonumber\\
		&=&u'\widehat{V}+\overline{\nabla^{\mathbb{R}^n}_V W}+a_0\nabla_{\partial_t} \widehat{W}.
	\end{eqnarray}
	
	Taking into account that the first term of (\ref{eq10}) is tangent to $S$, for any $\eta\in\Gamma(\mathrm{An}(dt))$ verifying $\eta\in\Gamma(\mathcal{R}^\perp)$, the second fundamental form associated with $\overline{\eta}$ satisfies 
	\begin{eqnarray}
		\mathrm{II}_{\eta}(X,Y)&=&b\,h(\nabla_X (\partial_t+\widehat{W}),\eta)=b\,u'\,h(\widehat{V},\eta)+b\,h(\nabla^{\mathbb{R}^n}_V W,\eta)+a_0b\,h(\widehat{W} ,\eta)\nonumber\\&=&bf\,f'\langle V,\eta\rangle+bf^2\,\mathrm{II}_\eta^{U}(V,W)+a_0b\,\partial_t\big(h(P^{\partial_t}Y,\eta)\big)-a_0b\,u'\,h\big(\widehat{W},\eta\big)\nonumber\\
		&=&bf\,f'\langle V,\eta\rangle+a_0b\,f^2\langle W,\eta\rangle\,\partial_t\big(f\,\mathrm{log}(f\,\langle W,\eta\rangle\big).
	\end{eqnarray}
	Therefore, $S$ is totally geodesic in $M$ if and only if $P^{\partial_t}X\perp \eta$ and $P^{\partial_t}Y\perp \eta$, i.e., $\partial_t$ must be a section of $TS$.
	
\end{proof}

Let $\gamma:J\subset\mathbb{R}\longrightarrow M$ be a proper UD observer. Making use of (\ref{eq:global expression G}), $P^{\partial_t}\gamma'=P+\alpha\,N$, being $P:=P^{\partial_t}L$, and then
\begin{equation}\label{gammaprima}
	\gamma'=\partial_t|_{\gamma}+P+\alpha\,N.
\end{equation}
Taking the usual coordinate basis of $\mathbb{R}^n$, $\lbrace \partial_{x_i}\rbrace$, and the local coordinates $(t,x^{i})$ in $M$, we may express $N=N^{i}\partial_{x_i}$. Thus, using (\ref{eq:GR conditions}) and (\ref{eq:GRW connection}),
$$0=\frac{D N}{dt}=(\partial_tN^{i}+u'\,N^{i})\partial_{x_i}.$$
As a consequence, $N(t)=\frac{f(t_0)}{f(t)}\,\overline{w}$, where $w\in\mathbb{R}^n$, and $t_0\in J$.

On the other hand, we have the following
\begin{equation*}
		\frac{DP}{dt}=(\partial_t P^{i}+u'\,P^{i})\partial_{x_i},\qquad
		\frac{DP}{dt}=-\frac{D}{dt}\partial_{t}|_{\gamma}=-u'\,(P+\alpha\,N).
\end{equation*}
By equating both expressions and integrating the resulting equation, we obtain
$$P(t)=\frac{f(t_0)}{f(t)}\left(f(t_0)\overline{v}-\bigg(\int_0^{t}f'(s)\alpha(s)ds\bigg)\,\overline{w}\right),\qquad v,w\in\mathbb{R}^n.$$
From (\ref{gammaprima}) and the previous result, we can state that $\gamma(J)\subset S$, where $S$ is a subset of a $k$-plane, ($k=1,2\,\mathrm{or}\,3$),
$$S:=\gamma(t_0)+\mathrm{Span}\left(\lbrace \partial_t,\overline{v},\overline{w}\rbrace\right).$$

\subsection{Relativistic case}
Let $(M=I\times_f U,g)$, $U\subset\mathbb{R}^n$, be a Lorentzian Robertson-Walker spacetime, and $\gamma:J\subset\mathbb{R}\longrightarrow M$ a non-free falling proper UD observer.
Consider $\mathcal{S}$ the distribution along $\gamma$ defined as follows,
$$\mathcal{S}(t) := \mathrm{Span}\left(\{\partial_t|_{\gamma(t)},\, L(t),\, N(t)\}\right)
= \mathrm{Span}\left(\{\partial_t|_{\gamma(t)},\, P^{\partial_t} L(t),\, P^{\partial_t} N(t)\}\right).$$

\noindent Since $P^{\partial_t} L \in \mathfrak{X}(\gamma)$ and $P^{\partial_t} L \perp \partial_t|_{\gamma}$,
we can find vector fields $V,W\in\mathfrak{X}(U)$ such that $P^{\partial_t} L=\overline{V}$ and $P^{\partial_t} N=\overline{W}$.
Hence, using suitable extensions,
\[
\nabla_{P^{\partial_t} L}\partial_t
= u'\,P^{\partial_t} L,
\qquad
\nabla_{P^{\partial_t} N}\partial_t
= u'\,P^{\partial_t} N.
\]

\noindent Therefore, using (\ref{connection}),
\begin{align*}
	\frac{D}{dt}\partial_t|_{\gamma}
	&= \cosh(\alpha)\,\nabla_L \partial_t
	+ \sinh(\alpha)\,\nabla_N \partial_t 
	= \cosh(\alpha)\,u'\,P^{\partial_t} L
	+ \sinh(\alpha)\,u'\,P^{\partial_t} N \\
	&= u'\,P^{\partial_t} \gamma'(t) \in \mathcal{S}(t).
\end{align*}

\noindent Consequently, the distribution $\mathcal{S}(t)$ is preserved by parallel transport
along $\gamma$.

Now, consider the orthogonal vector bundle \( \mathcal{S}^\perp \) and a local orthonormal frame  
\(\{\chi_i(t)\}_{i=1}^{\ell}\) expanding \( \mathcal{S}^\perp(t) \) along the curve \( \gamma \), and such that  
\[
\frac{D \chi_i}{dt} = 0, \quad \forall i = 1, \ldots, \ell.
\]

\noindent Since \( \chi_i \perp \partial_t \), we can choose local coordinates \( (t, x^i) \) such that  
$\chi_i = \sum_{j=1}^{n} a_{ij} \partial_{j}|_\gamma$, $a_{ij}\in\mathcal{C}^\infty(J)$. Thus,  
\[
0 = \frac{D \chi_i}{dt} = \sum_{j} ( a'_{ij} + u' a_{ij} ) \partial_{j}|_\gamma  
\Rightarrow a_{ij}(t) = \frac{c_{ij}}{f(t)}, \quad c_{ij} \in \mathbb{R}.
\]
\noindent Therefore,  
\[
\chi_i(t)= \frac{1}{f(t)} C_i,\quad \quad C_i \in \mathbb{R}^n , \qquad i = 1, \ldots, \ell.
\]
\noindent Consequently, we have

\[
\mathcal{S}^\perp(t) = \mathrm{Span} \left(\left\{ \frac{1}{f(t)} C_i \right\}_{j=1}^{\ell}\right) = \text{Span}\left( \left\{ C_j\right\}_{j=1}^{\ell}\right), \quad C_j \in \mathbb{R}^n.
\]
Hence, \( \mathcal{S}(t) \) is orthogonal to \( \ell \) linearly independent constant vector fields in \( \mathbb{R}^n \),  
so the projection onto the fiber yields an \( (n - \ell) \)-dimensional plane. Then, the submanifold integrating \( \mathcal{S}(t) \) is an \( (n + 1 - \ell) \)-dimensional plane in $I \times U \subset\mathbb{R}^{n+1}$, which can be expressed as  
\[
\mathcal{S}(t) = \mathrm{Span} \left(\{ \partial_t, K_1, K_2 \}\right),\qquad K_1,K_2\in\mathbb{R}^n.
\]

By means of formulae (\ref{connection}), one can easily verify that the subset of a \(k\)-plane \(S\) (\(k =1, 2\) or \(3\)) is a timelike totally geodesic submanifold in \(M\).

\backmatter

%
%
%

%
%

\section*{Funding}
The three authors are partially supported by Spanish MICINN project PID2021-126217NB-I00. In addition, this research was supported by an FPU grant from the Spanish MICIU to the third author.

\section*{Author contribution}
All authors have accepted responsibility for the entire content of this manuscript and approved its submission. 

\section*{Data availability}
Not applicable.

\end{document}